\documentclass[a4paper,12pt,epsfig]{article}

\usepackage{amsmath,amssymb,amsthm}
\usepackage{graphics}

\usepackage{wrapfig}
\usepackage{hyperref}

\usepackage{tikz}

\newtheorem{theorem}{Theorem}[section]
\newtheorem{lemma}[theorem]{Lemma}
\newtheorem{corr}[theorem]{Corollary}
\newtheorem{df}[theorem]{Definition}

\newcommand{\so}{\mathop{\rm SO}\nolimits}

\newcommand{\card}{\mathop{\rm card_{\mathcal F}}\nolimits}
\newcommand{\dist}{\mathop{\rm dist}\nolimits}

\title {Analogs of Steiner's porism and Soddy's hexlet  in higher dimensions via spherical codes}

\author {Oleg R. Musin\thanks{The author is partially supported by the NSF grant DMS-1400876 and the RFBR grant 15-01-99563.}}

\begin{document}
\date{}
\maketitle

\begin{abstract} In this paper we consider generalizations to higher dimensions of classical results on chains of tangent spheres.  
\end {abstract}

\medskip

\noindent {\bf Keywords:} Steiner's porism, Soddy's Hexlet, spherical codes 

\medskip 

\section{Introduction}

Suppose we have a chain of $k$ circles all of which are tangent to two given non-intersecting circles $S_1$, $S_2$, and each circle in the chain is tangent to the previous and next circles in the chain. Then, any other circle C that is tangent to $S_1$ and $S_2$ along the same bisector is also part of a similar chain of $k$ circles. This fact is known as {\em Steiner's porism} \cite[Chap. 7]{Barnes}, \cite[Chap. 4, 5]{Ogilvy}.  The usual proof of this is simply to choose an inversion that makes $S_1$ and $S_2$ concentric, after which the result follows immediately by rotation symmetry.  (Below are shown two closed Steiner chains and the inversion transform to a chain of congruent circles.) 

\medskip

\medskip


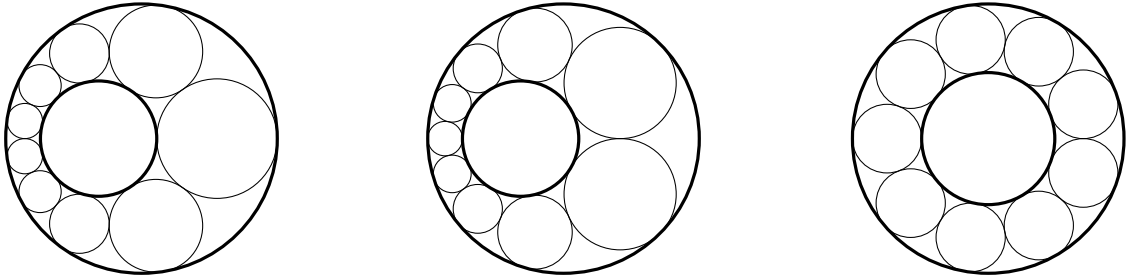
\begin{tikzpicture}[scale=1.5]


\draw [very thick](0.8762,0) circle (1.1905);
\draw [very thick](0.5,0) circle (0.5099);

\draw ( 1.5383, 0) circle (0.528382);
\draw ( 1.0026 , 0.7703) circle ( 0.4099 );
\draw (  0.3316,  0.7527) circle (0.2614);
\draw ( -0.0131 ,0.4688  ) circle (0.1851 );
\draw ( -0.1470 ,0.1555 ) circle (0.1555 );
\draw ( -0.1470 ,-0.1555 ) circle (0.1555 );
\draw (  -0.0131  , -0.4688 ) circle ( 0.1851);
\draw (   0.3316 ,  -0.7527 ) circle (0.2614);
\draw (   1.0026, -0.7703 ) circle ( 0.4099);


\draw [very thick](4.5762,0) circle (1.1905);

\draw [very thick](4.2,0) circle (0.5099);

\draw (5.0726, 0.4917) circle (0.4917 );
\draw ( 4.3261,  0.8269 ) circle ( 0.3266);
\draw ( 3.8246  , 0.6209 ) circle (0.2156);

\draw (  3.6004, 0.3120 ) circle (0.1660 );
\draw (  3.5379, 0 ) circle ( 0.1522);
\draw (  3.6004 , -0.3120 ) circle ( 0.1660);

\draw (  3.8246, -0.6209 ) circle ( 0.2156);
\draw (4.3261, -0.8269 ) circle ( 0.3266);
\draw ( 5.0726,-0.4917 ) circle (0.4917);

\draw [very thick](8.3,0) circle (1.1905);
\draw [very thick](8.3,0) circle (0.5837);

\draw ( 9.1336 ,0.3034  ) circle (0.3034);

\draw (8.7435 ,0.7682 ) circle (0.3034);
\draw ( 8.1460, 0.8736 ) circle (0.3034);
\draw ( 7.6204, 0.5702 ) circle (0.3034);
\draw (  7.4129,0 ) circle (0.3034);

\draw ( 7.6204, -0.5702 ) circle (0.3034);
\draw (8.1460, -0.8736 ) circle (0.3034);
\draw ( 8.7435,   -0.7682 ) circle (0.3034);
\draw ( 9.1336,   -0.3034 ) circle (0.3034);

\end{tikzpicture}




\medskip

\medskip

{\em Soddy's hexlet} is a chain of six spheres each of which is tangent to both of its neighbors and also to three mutually tangent given spheres. Frederick Soddy published the following theorem in 1937 \cite{Soddy}: 
%
%
\noindent{\em It  is always possible to find a hexlet for any choice of three mutually tangent spheres.}  
%
%
\noindent Note that Soddy's hexlet was also discovered independently in Japan, as shown by Sangaku tablets from 1822 in the Kanagawa prefecture \cite{Jap}. 

The general problem of finding a hexlet for three given mutually tangent spheres $S_1$, $S_2$, and $S_3$, can be reduced to the annular case using inversion. Inversion in the point of tangency between spheres $S_1$ and $S_2$ transforms them into parallel planes $P_1$ and $P_2$.  Since sphere $S_3$ is tangent to both $S_1$  and $S_2$ and does not pass through the center of inversion, $S_3$ is transformed into another sphere $S'_3$ that is tangent to both planes.  Six spheres may be packed around $S'_3$ and touches planes $P_1$ and $P_2$. Re-inversion restores the three original spheres, and transforms these six spheres into a hexlet for the original problem \cite{Barnes,Ogilvy}.

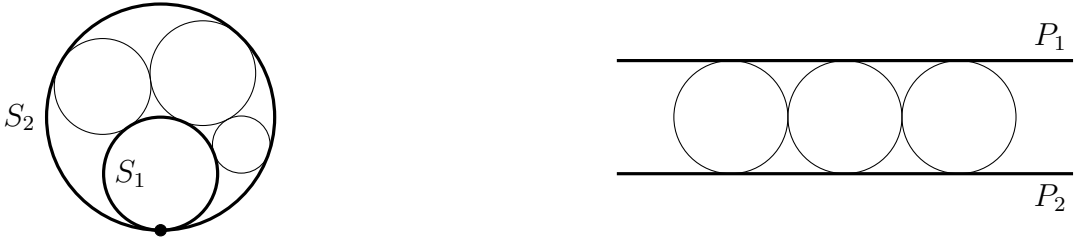
\medskip

\begin{tikzpicture}[scale=1.5]


\draw [very thick] (0,0.5) circle (0.5);
\draw [very thick](0,1) circle (1);

\draw [fill] (0,0) circle [radius=0.05];

\draw ( 0.3704, 1.3889) circle (0.4630 );
\draw (-0.5085, 1.2712  ) circle ( 0.4237);
\draw (0.7071,0.7576  ) circle (0.2525);

\node [right] at (-0.5, 0.5) {$S_1$};
\node [left] at (-1, 1) {$S_2$};

\draw [very thick] (4,1/2) -- (8,1/2);
\draw [very thick] (4,3/2) -- (8,3/2);

\node [above] at (7.8, 1.5) {$P_1$};
\node [below] at (7.8, 1/2) {$P_2$};

\draw ( 6, 1) circle (1/2);
\draw ( 7, 1) circle (1/2);
\draw ( 5, 1) circle (1/2);

\end{tikzpicture}

\medskip
 





 

  Let ${\mathcal F}:=\{S_1,S_2\}$, where $S_1$ and $S_2$ are tangent spheres in ${\mathbb R}^{n}$. Let $\Pi_n({\mathcal F})$ denote the set of all  (non--congruent) sphere packings in ${\mathbb R}^{n}$ such that all spheres in a packing $P\in \Pi_n({\mathcal F})$ are tangent to both spheres from ${\mathcal F}$.  In  \cite{HK2001} the authors report that there is an unpublished result by Kirkpatrick and  Rote about this case.   In fact, they proved that 
 
 \medskip 
  {\em There is a one--to--one correspondence $T_{\mathcal F}$ between sphere packings from $\Pi_n({\mathcal F})$ and unit sphere packings in ${\mathbb R}^{n-1}$.}
  
 \medskip
 
\noindent It is easy to prove. Indeed, let $T_{\mathcal F}$ be an inversion in the point of tangency between spheres from ${\mathcal F}$ such that it makes $S_1$ and $S_2$ parallel hyperplanes with the distance between them equals 2. Then the result follows immediately by the fact that a packing $P\in \Pi_n({\mathcal F})$ transforms to a unit sphere packing $T_{\mathcal F}(P)$.  (\cite[Proposition 4.5]{HK2001} contains a sketch of proof.)  


  
\medskip

Let $X$ be a set of points in a unit sphere ${\mathbb S}^{d-1}$. We say that $X$ is a {\em spherical $\psi$--code} 
if the angular distance between distinct points in $X$ is at least $\psi$.  Denote by $A(d,\psi)$  the maximal size of a $\psi$--code in ${\mathbb S}^{d-1}$ \cite{CS}.

Note that $A(d,\pi/3)=k(d)$, where  by $k(d)$ we denote the {\em kissing number}, i.e. the maximum number of  non--overlapping unit spheres in ${\mathbb R}^{d}$ that can be arranged so that all of them touch one (central) unit sphere.


In this paper we show a relation between sphere packings in ${\mathbb R}^{n}$ that are tangent spheres in a given family $ {\mathcal F}$ and spherical codes (Theorem \ref{th1}). This relation gives generalizations of Steiner's porism and Soddy's hexlet  to higher dimensions.


\section{${\mathcal F}$--kissing arrangements and spherical codes}

Here we say that two distinct spheres $S_1$ and $S_2$ in ${\mathbb R}^{n}$ are {\em non--intersecting} if  the intersection of these spheres is not a sphere of radius $r>0$. In other words, either $S_1\cap S_2=\emptyset$ or these spheres  touch each other.

\begin{df} \label{def21}
Let ${\mathcal F}=\{S_1,\ldots,S_m\}$ be a family of $m$ arbitrary spheres in ${\mathbb R}^{n}$. (Actually,  $S_i$ can be a sphere of any radius or a hyperplane.)  We say that a set ${\mathcal C}$ of spheres in ${\mathbb R}^{n}$ is an {${\mathcal F}$--kissing arrangement} if \\
(1) each sphere from ${\mathcal C}$ is tangent to all spheres from ${\mathcal F}$; \\
(2) each sphere from ${\mathcal C}$ is tangent to at least one sphere from  ${\mathcal C}$; \\
 (3) any two distinct spheres from ${\mathcal C}$  are non--intersecting. 
 \end{df}
 
It is clear that if ${\mathcal C}$ is nonempty and one of spheres from ${\mathcal F}$ contains another then all $S_i$ as well as all spheres from ${\mathcal C}$ lie in this sphere. If there are no such sphere in ${\mathcal F}$, then depending of radii and arrangements of $S_i$ either one of spheres from ${\mathcal C}$ contains all other from ${\mathcal C}$ and ${\mathcal F}$, or all spheres in ${\mathcal C}$ are non--overlapping.

\begin{df} \label{defS}
Let ${\mathcal F}=\{S_1,\ldots,S_m\}$, $m\ge 2$, be a family of $m$ spheres in ${\mathbb R}^{n}$. We say that ${\mathcal F}$ is an S--family if \\
(1) $S_1$ and $S_2$ are non--intersecting spheres; \\
(2) each $S_i$ with $i>2$ can intersect at most one $S_j$ with $j=1,2$; \\
(3) there are non-empty  ${\mathcal F}$--kissing arrangements and all of them are finite.
 \end{df}


 \noindent{\bf Remark.} I wish to thank the anonymous referee of this paper  who pointed out that if Definition \ref{defS} has only assumptions (1) and (3), then ${\mathcal F}$--kissing arrangements are possible can have spheres that touch some spheres in ${\mathcal F}$ {\em from the outside} and  some {\em from the inside.}  
 
 Consider the following example. Let ${\mathcal F}:=\{S_1,S_2,S_3\}$, where  $S_1$ and $S_2$ be two concentric spheres (or two parallel hyperplanes) in ${\mathbb R}^{n}$. Let $S_3$ be a sphere that intersects $S_1$ and $S_2$. Then for some cases there are ${\mathcal F}$--kissing spheres such that some of them are tangent to $S_3$ from the outside and some from the inside. 
 
However, if we have (2), then there is at most one sphere that is tangent  to $S_1$, $S_2$, and $S_3$ from the inside. Indeed, suppose $S_3$ intersects $S_1$. Then Definition \ref{defS}(2) yields that $S_3$ either has no common points with  $S_2$  or $S_3$ is tangent to $S_2$ at some point $p$. In the first case there are no ${\mathcal F}$--kissing spheres that are tangent to $S_3$ from the inside. It is easy to see that in the second case we can have at most one sphere that is tangent to $S_2$ and $S_3$ at $p$.  By Definition \ref{def21}(2) this sphere cannot be a sphere in the ${\mathcal F}$--kissing arrangement. 
 
 
 \medskip

Note that in Steiner's chain problem, ${\mathcal F}$ consists of  two non-intersecting circles $S_1$ and $S_2$, and in the problem of finding a hexlet, ${\mathcal F}$ consists of three mutually tangent spheres $S_1$, $S_2$, and $S_3$. Now we consider a general case. 


\begin{theorem} \label{th1}  Let ${\mathcal F}=\{S_1,\ldots,S_m\}$, $2\le m<n+2$, be an S--family of $m$ spheres in ${\mathbb R}^{n}$.  Then there is a one--to--one correspondence $\Phi_{\mathcal F}$ between ${\mathcal F}$--kissing arrangements and spherical $\psi_{\mathcal F}$--codes in ${\mathbb S}^{d-1}$, where $d:=n+2-m$ and the value $\psi_{\mathcal F}$ is uniquely defined by the family ${\mathcal F}$. 
\end{theorem}

\begin{proof}  There are two cases: (i) $S_1$ and $S_2$ are tangent or (ii) $S_1$ and $S_2$ do not touch each other. In the first case  let $O$ be the contact point of these spheres and if we apply the sphere inversion $T$ with center $O$ and an arbitrary radius $\rho$, then  $S_1$ and $S_2$  become two parallel hyperplanes $S'_1$ and $S'_2$. In case (ii) we can use the famous theorem: {\em It is always possible to invert $S_1$ and $S_2$ into a pair of concentric spheres   $S'_1$ and $S'_2$} (see \cite[Theorem 13]{Ogilvy}).

 Let $P$ be an  ${\mathcal F}$--kissing arrangement. Since all  spheres from $P$ touch $S_1$ and $S_2$ after the inversion they become  spheres that touch $S'_1$ and $S'_2$. In both cases that yields that all spheres from $P':=T(P)$ are congruent.   Without loss of generality we can assume that spheres from $P'$ are unit. Thus we have a unitt sphere packing $P'=\{C'_j\}$  in ${\mathbb R}^{n}$  such that each sphere $C'_j$ from $P'$ is tangent to all $S'_i:=T(S_i)$, $i=1,\ldots,m$. 
 
In case (i) denote by $Z_0$ the hyperplane of symmetry of $S'_1$ and $S'_2$ and in case (ii) $Z_0$ be a sphere  of radius $(r_1+r_2)/2$ that is concentric with $S'_1$ and $S'_2$, where $r_i$ is the radius of $S'_i$.   If $m>2$, let $Z_i$, $i=3,\ldots,m$,  denote a sphere  of radius $(r_i+1)$ that is concentric with $S'_i$.  Let $S_{\mathcal F}$ be the locus of centers of spheres that are tangent to all $S'_i$. If $m=2$, the $S_{\mathcal F}=Z_0$ and for $m>2$,  $S_{\mathcal F}$ is the intersection of spheres $Z_0$ and $Z_i$, $i=3,\ldots,m$.  


Note that by assumption $S_{\mathcal F}$ is not empty. Moreover, since   all ${\mathcal F}$--kissing arrangements are finite, $S_{\mathcal F}$ is a sphere of radius $r> 0$.

Since all $C_j$ are unit sphere, the distance between centers of distinct spheres in $P'$ is at least 2. Therefore, if $r<1$, then $P$ contains just one sphere. In this case put for $\psi_{\mathcal F}$ any number greater than $\pi$. 
 
Now consider the case when $S_{\mathcal F}$ is a $(d-1)$--sphere of radius $r\ge1$ Let $\psi_{\mathcal F}$ be the angular distance between centers in $S_{\mathcal F}$  of two tangent unit spheres in ${\mathbb R}^{n}$. In other words, $\psi_{\mathcal F}$ is the angle between equal sides in an isosceles triangle with side lengths $r,\, r,$ and 2.   We have 
$$
\psi_{\mathcal F}:=\arccos\left(1-\frac{2}{r^2}\right). 
$$
 
Let $f:S_{\mathcal F}\to U_{\mathcal F}$ be the central projection, where $U_{\mathcal F}$ denotes a unit sphere that is concentric with $S_{\mathcal F}$.  Denote $c_P$ the set of centers of $C_j$. Let $X:=f(c_P)$. Then  $X$ is a spherical $\psi_{\mathcal F}$ --code in ${\mathbb S}^{d-1}$.  

Let $X$ be any spherical $\psi_{\mathcal F}$ --code in ${\mathbb S}^{d-1}\simeq U_{\mathcal F}$. Then we have a unit sphere packing $Q_X$ with centers in $c_X:=f^{-1}(X)$ such that each sphere  from $Q_X$ is tangent to all $S'_i$. It is clear that $P:=T(Q_X)$ is an ${\mathcal F}$--kissing arrangement. 

Thus, a one--to--one correspondence $\Phi_{\mathcal F}$ between ${\mathcal F}$--kissing arrangements and spherical $\psi_{\mathcal F}$--codes in ${\mathbb S}^{d-1}$ is well defined.  This completes the proof. 
\end{proof}


\begin{corr} \label{cor22}  Let ${\mathcal F}=\{S_1,\ldots,S_m\}$ be an S--family of  spheres in ${\mathbb R}^{n}$.   Denote by $\card$ the maximum cardinality of  ${\mathcal F}$--kissing arrangements.  Then 
$$
\card=A(d,\psi_{\mathcal F}).  
$$
In particular, $\card\ge d+1$ if and only if  $\psi_{\mathcal F}\le\arccos(-1/d)$. 
\end{corr} 
\begin{proof} The  equality $\card=A(d,\psi_{\mathcal F})$ immediately follows  from Theorem \ref{th1}. Since $a_d:=\arccos(-1/d)$ is the side length of  a regular spherical $d$--simplex in ${\mathbb S}^{d-1}$, we have $A(d,a_d)=d+1$. Thus, if $\psi\le a_d$, then $A(d,\psi)\ge d+1$.
 \end{proof}

Theorem \ref{th1} states  a one--to--one correspondence between ${\mathcal F}$--kissing arrangements and spherical codes.  We say that {\em two ${\mathcal F}$--kissing arrangements $M$ and $N$ are   equivalent} if the correspondent spherical  $\psi_{\mathcal F}$--codes $X$ and $Y$ are isometric in ${\mathbb S}^{d-1}$. 

\begin{theorem}  \label{th2} Let ${\mathcal F}=\{S_1,\ldots,S_m\}$ be an S--family of  spheres in ${\mathbb R}^{n}$.  Then any $\psi_{\mathcal F}$--code $X$ in\,   ${\mathbb S}^{d-1}$ uniquely determines the set of equivalent  ${\mathcal F}$--kissing arrangements $\{P_A(X)\}$ such that this set can be parametrized by $A\in \so(d)$.  Moreover, for any isometric  $\psi_{\mathcal F}$--codes $X$ and $Y$ in ${\mathbb S}^{d-1}$  and $A,B\in \so(d)$,  $P_A(X)$ can be transform to $P_B(Y)$ by a conformal map. 
\end{theorem}
\begin{proof} Here we use the same notations as in the proof of Theorem \ref{th1}. 

Denote $P=T(Q_X)$ by $P_I$, where $I$ is the identity element in $\so(d)$. If $\psi_{\mathcal F}$--codes $X$ and $Y$ are isometric in ${\mathbb S}^{d-1}$, then there is $A\in \so(d)$ such that $Y=A(X)$.  Denote  $T(Q_A)$ by $P_A$.  We have 
$$P_A=h_A(P_I), \quad h_A:=T\circ A\circ T.$$  
It is clear that $h_A$ is a conformal map. 
\end{proof}


\section{Analogs of Steiner's porism and Soddy's hexlet} 

{\subsection{Analogs of Steiner's porism.} 
Theorem \ref{th2} can be considered as a generalization of Steiner's porism. For a given family ${\mathcal F}$ and spherical $\psi_{\mathcal F}$--code $X$ in ${\mathbb S}^{d-1}$ there are ${\mathcal F}$--kissing arrangements that  correspondent to $X$. 

However, Steiner's porism has stronger property. A Steiner chain is formed from one starting circle and {\em each circle in the chain is tangent to the previous and next circles in the chain}. 
If the last circle touches the first, this will also happen for any position of the first circle. Thus, a position of the first circle uniquely determines a Steiner chain. 

Now we extend  this property to higher dimensions.   We say that an ${\mathcal F}$--kissing arrangement ${\mathcal C}=\{C_1,\ldots,C_k\}$  is a {\em $k$--clique} if all spheres in ${\mathcal C}$ are mutually tangent. We say that a sphere $C_{k+1}$ is {\em adjacent} to  ${\mathcal C}$ if $C_{k+1}$ is tangent to all spheres of ${\mathcal C}$ and ${\mathcal F}$. 

\begin{lemma} \label{L31} Let ${\mathcal F}=\{S_1,\ldots,S_m\}$ be an S--family of  spheres in ${\mathbb R}^{n}$  with $\card\ge d+1.$  Then the set of $(d-1)$-cliques is not empty and for any $(d-1)$-clique ${\mathcal C}$  there are exactly two adjacent spheres.
\end{lemma}
 \begin{proof} Corollary \ref{cor22} yields that $\psi_{\mathcal F}\le\arccos(-1/d)$. Therefore, a regular spherical $(d-2)$--simplex of side length $\psi_{\mathcal F}$ can be embedded into ${\mathbb S}^{d-1}$. For this simplex in  ${\mathbb S}^{d-1}$ there are exactly two possibilities to complete it to regular spherical $(d-1)$--simplices. By Theorem \ref{th1} these two new vertices correspond to two  adjacent spheres.
 \end{proof}

 
Now we define a Steiner arrangement for all dimensions. First we define a tight  ${\mathcal F}$--kissing arrangement, where ${\mathcal F}$ is an S--family of  spheres in ${\mathbb R}^{n}$.  Let ${\mathcal C}_0$ be any $(d-1)$--clique.  By Lemma \ref{L31} there are two adjacent spheres for  ${\mathcal C}_0$. Let $C_1$ be one of them.  Then ${\mathcal C}_1:={\mathcal C}_0\cup C_1$ is a $d$--clique of tangent spheres. Suppose that after $k$ steps we have an ${\mathcal F}$--kissing arrangement ${\mathcal C}_k$. We can do the next step only if in ${\mathcal C}_k$ there are a $(d-1)$--clique and its adjacent sphere $C_{k+1}$ such that  ${\mathcal C}_{k+1}:={\mathcal C}_k\cup C_{k+1}$ is an ${\mathcal F}$--kissing arrangement.  Denote by $t$ the maximum number of possible steps. It is clear, $t\le\card$.  We call ${\mathcal C}_t$ a {\em tight ${\mathcal F}$--kissing arrangement}. 

\medskip

Note that for $d=2$ a tight chain ${\mathcal C}_t$ is Steiner if the first circle of the chain touch the last one. It can be extended for all dimensions. We say that a  tight ${\mathcal F}$--kissing arrangement  ${\mathcal C}_t$ is {\em Steiner} if ${\mathcal C}_t$ contains all adjacent spheres of all its  $(d-1)$--cliques. Equivalently, an ${\mathcal F}$--Steiner arrangement can be define by the following way. 

\begin{df} \label{dfS} Let ${\mathcal F}=\{S_1,\ldots,S_m\}$ be an S--family of  spheres in ${\mathbb R}^{n}$  with $\card\ge d+1.$ 
 We say that an ${\mathcal F}$--kissing arrangement ${\mathcal C}$ is Steiner if  it contains a $(d-1)$--clique and for all  $(d-1)$--cliques in ${\mathcal C}$ their adjacent spheres also lie in ${\mathcal C}$. 
  \end{df}

Recall that a simplicial polytope is a polytope whose facets are all simplices.

\begin{df}  Let ${\mathcal F}=\{S_1,\ldots,S_m\}$ be an S--family of  spheres in ${\mathbb R}^{n}$. 
An ${\mathcal F}$--kissing arrangement is called  $(d-1)$--{simplicial} if the convex hull of the  correspondent spherical code   in ${\mathbb S}^{d-1}$ is a $(d-1)$--simplicial regular polytope.  We denote this polytope by $P_{\mathcal F}$.  
 \end{df} 
 
\begin{lemma} \label{L32} Let ${\mathcal F}=\{S_1,\ldots,S_m\}$ be an S--family of  spheres in ${\mathbb R}^{n}$  with $\card\ge d+1.$  An ${\mathcal F}$--kissing arrangement  is Steiner if and only if it is $(d-1)$--simplicial.  
\end{lemma}


 \begin{proof} Clearly, if an ${\mathcal F}$--kissing arrangement ${\mathcal C}$  is simplicial then it is Steiner. Suppose ${\mathcal C}$  is Steiner. Then the convex hull $P$ of the correspondent spherical $\psi_{\mathcal F}$--code $\Phi_{\mathcal F}({\mathcal C})$ (see Theorem \ref{th1}) is a polytope that has a $(d-2)$--face $P_0$ which is a regular $(d-2)$--simplex of side length $\psi_{\mathcal F}$. By Lemma \ref{L31}, $P_0$ has two adjacent points $v_0$ and $v_1$ in ${\mathbb S}^{d-1}$. Moreover, by Definition \ref{dfS}, these two  points are vertices of $P$. Therefore all vertices of a bipyramid $P_1:=v_0\cup P_0\cup v_1$ are vertices of $P$. It is clear that all facets of $P_1$  are regular $(d-2)$--simplices, i.e. they are $(d-2)$--cliques in ${\mathcal C}$. It yields that $P_1$ a sub--polytope of $P$. Next, we add all new adjacent vertices to  $(d-2)$--faces of $P_1$. We denote this sub--polytope of $P$ by $P_2$. We can continue this process and define new $P_i$. It is easy to see that after finitely many steps we obtain $P_k=P$.
 
Note that for $i>0$ any $P_i$ consists of regular $(d-1)$--simplices of side length $\psi_{\mathcal F}$. Then all faces of $P$ are regular simplices. Since $P$ is a spherical polytope, we have that $P$ is regular.    \end{proof}

 
 \begin{theorem} \label{prop}  Let ${\mathcal F}=\{S_1,\ldots,S_m\}$ be an S--family of  spheres in ${\mathbb R}^{n}$.  If for ${\mathcal F}$ there exists a Steiner arrangement then we have one of the following cases 
 \begin{enumerate}
 \item $d=2$, $\psi_{\mathcal F}=2\pi/k$, $k\ge3$, and $P_{\mathcal F}$ is a regular polygon with $k$ vertices. 
 
 \item  $\psi_{\mathcal F}=\arccos(-1/d)$ and $P_{\mathcal F}$ is a regular $d$--simplex with any $d\ge 2$. 
 
  \item  $\psi_{\mathcal F}=\pi/2$ and $P_{\mathcal F}$ is a regular $d$--crosspolytope with any $d\ge 2$. 
  
   \item $d=3$, $\psi_{\mathcal F}=\arccos(1/\sqrt{5})$ and $P_{\mathcal F}$ is a regular icosahedron. 
   
     \item $d=4$, $\psi_{\mathcal F}=\pi/5$ and $P_{\mathcal F}$ is a regular $600$--cell. 
 \end{enumerate}
 \end{theorem}
\begin{proof} Lemma \ref{L32} reduces a classification of Steiner arrangements to an enumeration of simplicial  regular polytopes.  The list of these polytopes is well known, see \cite{Cox}, and it is as in the theorem. 
 \end{proof}



In particular, Theorem \ref{prop} shows that for any of these five cases all  tight ${\mathcal F}$--kissing arrangements are equivalent. The following corollary is a generalization of Steiner's porism.  

\begin{corr}  Let ${\mathcal F}$ be an S--family of  spheres in ${\mathbb R}^{n}$.  If there is an  ${\mathcal F}$--Steiner arrangement then any tight ${\mathcal F}$--kissing arrangement is Steiner. 
\end{corr}


\subsection{Analogs of Soddy's hexlet}  Soddy \cite{Soddy} proved that for any family  ${\mathcal F}$ of three mutually tangent spheres in ${\mathbb R}^{3}$ there is a chain of six spheres (hexlet) such that each sphere from this chain is tangent all spheres from  ${\mathcal F}$. Now we extend this theorem to higher dimensions.

Let $m\ge2$. Denote 
$$\psi_m:=\arccos\left(\frac{1}{m-1}\right).$$


\begin{theorem} \label{th31} Let  $3\le m<n+2$. Let $X$ be a spherical $\psi_m$--code in ${\mathbb S}^{d-1}$, where $d:=n+2-m$.  Then for any family ${\mathcal F}$  of $m$ mutually tangent spheres in ${\mathbb R}^{n}$ there is an ${\mathcal F}$--kissing arrangement that corresponds to $X$. 
\end{theorem}

\begin{proof} Here we use the same notations as in the proof of Theorem \ref{th1}. 

This theorem follows from Theorem \ref{th1} using the fact that $\psi_{\mathcal F}=\psi_m$.  Indeed, we have case (i) and therefore  $S'_1$ and $S'_2$ are two parallel hyperplanes. Since the spheres in ${\mathcal F}$  are  mutually tangent, we have that all $S'_i$, $i=3,\ldots,m$, are unit spheres. It is not hard to prove that $S_{\mathcal F}$ is the intersection of $(m-2)$  spheres of radius 2 centered at points  $C=\{c_3,\ldots,c_m\}$ in ${\mathbb R}^{n-1}$ such that if $m>3$ then $\dist(c_i,c_j)=2$ for all  distinct $c_i$ and $c_j$ from $C$. Then  
$$
r=\sqrt{\frac{2m-2}{m-2}} \; \mbox{ and } \: 1-\frac{2}{r^2}=\frac{1}{m-1}.
$$
It proves the equality $\psi_{\mathcal F}=\psi_m$. 
\end{proof}


Let   ${\mathcal F}$ be a family of $m$ mutually tangent spheres in ${\mathbb R}^{n}$.  Denote $$S(n,m):=\card.$$ 
Corollary \ref{cor22} and Theorem \ref{th31} imply  
\begin{corr} 
$S(n,m)=A(n+2-m,\psi_m).$ In particular, $S(n,3)=k(n-1)$. 
\end{corr}

\noindent{\bf Examples.} 
\begin{enumerate}
\item If $m=n+1$, then $\psi_m=\pi$. It implies that $S(n,n+1)=2$. Actually, this fact can be proved directly, there are just two spheres that are tangent to
$n+1$ mutually tangent spheres in ${\mathbb R}^{n}$. 

\item Now consider a classical case $m=n=3$. We have  $$S(3,3)=A(2,\pi/3)=k(2)=6.$$ Then a maximum $\pi/3$--code in  ${\mathbb S}^{1}$ is a regular {\em hexagon.} The corresponding ${\mathcal F}$-sphere arrangement is  a Soddy's hexlet. 

\item Let $m=3$ and $X$ be a spherical code of maximum cardinality $|X|=k(d)$, where $d:=n-1$. Then $X$ is a {\em kissing arrangement} (maximum $\pi/3$--code) in  ${\mathbb S}^{d-1}$. Note that the kissing number problem has been solved only for $n\le 4$, $n=8$ and $n=24$ (see   \cite{BDM,CS,Mus24}). However,  in several dimensions many nice kissing arrangements are known, for  instance,  in dimensions 8 and 24 \cite{CS}.

If $n=4$, i.e. $d=3$, then $k(d)=12$. In this dimension there are infinitely many non--isometric kissing arrangements.  We think that the {\em cuboctahedron} with 12 vertices representing the positions of 12 neighboring spheres  can be a good analog of Soddy's hexlet in four dimensions.

In four dimensions the kissing number is 24 and the best known kissing arrangement is a regular 24--{\em cell} \cite{Mus24}.   (However, the conjecture about  uniqueness of this kissing arrangement is still open.) So in dimension five a nice analog of Soddy's hexlet is the 24--cell. 

\item By Theorem \ref{th31} for ${\mathcal F}$--kissing arrangements correspondent spherical codes have to have the inner product $=1/(m-1)$.  The book \cite{CS} contains a large list of such spherical codes.  Moreover, some of them are universally optimal \cite[Table 1]{Cohn}. All these examples give analogs of Soddy's hexlet in higher dimensions. 

\item Hao Chen \cite[Sect. 3]{Chan} considers sphere packings for some graph joins.  \cite[Table 1]{Chan} contains a large list of spherical codes that give generalizations of Soddy's hexlet. 
\end{enumerate}


\medskip

\medskip

\noindent{\bf Acknowledgment.} I wish to thank Arseniy Akopyan and Alexey Glazyrin for useful comments and references.

\medskip

\medskip

 O. R. Musin,  University of Texas Rio Grande Valley, School of Mathematical and
 Statistical Sciences, One West University Boulevard, Brownsville, TX, 78520, USA.

 {\it E-mail address:} oleg.musin@utrgv.edu

\end{document}